\documentclass[12pt,a4paper]{article}

\usepackage{amssymb,amsmath,amsthm}
\usepackage[english]{babel}
\newtheorem{thm}{Theorem}
\newtheorem{prop}[thm]{Proposition}
\newtheorem{cor}[thm]{Corollary}

\newtheorem{lem}[thm]{Lemma}

\newtheorem{claim}[thm]{Claim}

\newtheorem{conj}[thm]{Conjecture}

\usepackage{indentfirst}
\begin{document}

\title{Generalizations of the Tree Packing Conjecture}
\author{D\'aniel Gerbner$^{\rm a,}$\footnote{Research supported by OTKA NK 78439}
\and
Bal\'azs Keszegh$^{\rm a,b,*,}$\footnote{Research supported by Swiss National
Science Foundation, Grant No. 200021-125287/1}
\and
Cory Palmer$^{\rm a,*}$\footnote{Corresponding author; corypalmer@gmail.com}
\and
\small $^{\rm a}${\it Hungarian Academy of Sciences, Alfr\'ed R\'enyi Institute}\\[-0.8ex]
\small {\it of Mathematics, P.O.B. 127, Budapest H-1364, Hungary}\\
\small $^{\rm b}${\it Ecole Polytechnique F\'ed\'erale de Lausanne}\\[-0.8ex]
\small {\it  EPFL-SB-IMB-DCG, 1015 Lausanne, Switzerland}\\
}

\maketitle

\abstract{The Gy\'arf\'as tree packing conjecture asserts that any
set of trees with $2,3, \dots , k$ vertices has an (edge-disjoint) packing into the complete graph on $k$ vertices. Gy\'arf\'as and Lehel proved that the conjecture holds in some special cases. We address the problem of packing trees into $k$-chromatic graphs. In particular, we prove that if all but three of the trees are stars then they have a packing into any $k$-chromatic graph. We also consider several other generalizations of the conjecture.
}

\section{Introduction}\label{intro}
A set of (simple) graphs $G_1,G_2, \dots, G_k$ has a \emph{packing} into a graph $H$ if
$G_1,G_2, \dots, G_k$ appear as edge-disjoint subgraphs of $H$. In general we are concerned with
the case when each $G_i$ is a tree.
One of the best-known packing problems is the Tree Packing Conjecture (TPC) posed by Gy\'arf\'as \cite{GyLe}:

\begin{conj}[TPC]\label{theconjecture}
For $2 \leq i \leq n$, let $T_i$ be a tree on $i$ vertices. Then
the set of trees $T_2, \dots, T_n$ has a packing into the complete graph on $n$ vertices.
\end{conj}

A number of partial results related to the TPC have been found. The first results are by Gy\'arf\'as and Lehel \cite{GyLe} who proved that
the TPC holds with the additional assumption that all but two of the trees are stars.
Gy\'arf\'as and Lehel also showed that the TPC is true if each tree is either a path or a star. A second proof is by Zaks and Liu \cite{ZaLi}.
Bollob\'as \cite{Bo} showed that
the trees $T_2, \dots, T_s$ have a packing into $K_n$ if $s \leq n/\sqrt{2}$ and $T_i$ has $i$ vertices.
From the other side, Hobbs, Bourgeois and Kasiraj \cite{HoBoKa} showed that any three trees $T_n, T_{n-1}, T_{n-2}$ have a packing into $K_n$ if $T_i$ has $i$ vertices. A series of papers by Dobson \cite{Do1,Do2,Do3} concerns packing trees with some technical conditions.

Instead of packing trees into the complete graph, a number of papers have examined packing trees into complete bipartite graphs. Hobbs et al. \cite{HoBoKa} conjectured that the trees $T_2,\dots T_n$ have a packing into the complete bipartite graph
$K_{n-1, \lceil n/2 \rceil}$ if $T_i$ has $i$ vertices. The conjecture is true if each of the trees is a star or path. The case when $n$ is even was shown by Zaks and Liu \cite{ZaLi} and when $n$ is odd by Hobbs \cite{Ho}. Yuster \cite{Yu} showed that $T_2, \dots, T_s$ have a packing into $K_{n-1,\lceil n/2 \rceil}$ if $s \leq \lfloor \sqrt{5/8}n \rfloor$ and $T_i$ has $i$ vertices (improving the previously best-known bound by Caro and Roditty \cite{CaRo}).

Now we introduce a conjecture that would imply the TPC:

\begin{conj}\label{chromatic-conj}
For $2 \leq i \leq k$, let $T_i$ be a tree on $i$ vertices.
If $G$ is a $k$-chromatic graph, then the set of trees $T_2, \dots, T_k$ has a packing into $G$.
\end{conj}

The main result of the present paper concerns a special case of Conjecture \ref{chromatic-conj}.

\begin{thm}\label{main-thm}
For $2 \leq i \leq k$, let $T_i$ be a tree on $i$ vertices.
If $G$ is a $k$-chromatic graph and there are at most 3 non-stars among $T_2,\dots,T_k$, then they can be packed into $G$.
\end{thm}

Note that Theorem \ref{main-thm} can be stated in a stronger way as the proof only requires $G$ to have a subgraph that has a Grundy $k$-coloring (see e.g. \cite{ChSe}) and minimum degree $k-1$.
The immediate corollary of Theorem \ref{main-thm} for complete graphs was proved by Roditty \cite{Ro}\footnote{This proof contains some errors which have recently been corrected by the author
\cite{Ro1}.}.

\begin{cor}\label{three-cor}
The TPC is true with the additional assumption that all but three of the trees are stars.
\end{cor}

\section{Proof of Theorem \ref{main-thm}}\label{thm-proof}

Before moving to the proof let us introduce some additional definitions.

Let $x$ be a vertex with exactly one neighbor $y$ of degree greater than $1$
and at least one neighbor of degree $1$.
The induced substar $R$ spanned by $x$ and its neighbors of degree $1$ is called a \emph{pending star}.
The vertex $y$ will be referred to as the \emph{neighbor of $R$}.
A \emph{spider} is a tree that has a vertex whose removal results in isolated vertices and edges (i.e.
a spider is a graph with a central vertex and some branches of
length 1 or 2).

\begin{proof}
The proof will be by induction on $k$, but the precise form of the induction
depends on the structure of the largest trees. For $k \le 3$ the statement of the theorem is trivial.
Now let us assume that the statement of the theorem holds for all
values less than $k$.

Without loss of generality we can assume $G$ is a vertex-critical $k$-chromatic graph.
Thus $G$ has minimum degree at least $k-1$.
Let us choose a $k$-coloring of $G$ with color classes $A_1, A_2,
\dots, A_k$ such that any vertex $x \in A_i$ has a neighbor in
each color class $A_1, A_2, \dots, A_{i-1}$. Let $G_i=G \setminus
(A_1 \cup A_2 \cup \dots \cup A_{k-i})$ be the induced subgraph of
$G$ on the color classes $A_{k},A_{k-1}, \dots, A_{k-i+1}$. Note
that $G_i$ has chromatic number $i$.

For simplicity we will use edge-coloring terminology.
A \emph{partial edge-coloring} of a graph $G$ is an assignment of colors to some of the edges of $G$.
(We will omit the word ``partial.'')
An edge that receives no color is referred to as \emph{uncolored}.

We will construct an edge-coloring of $G$ such
that the subgraph consisting of the edges of color $i$ is isomorphic to the tree $T_i$. Clearly this edge-coloring problem is equivalent to packing the trees into $G$.

The proof is divided into several claims and cases according to the structure of the trees in $T_2,\dots, T_k$.
In each case we remove parts from $t\leq 3$ non-stars and delete $t$ stars from $T_2,\dots, T_k$ such that we are left
with a sequence of trees of order $2,\dots, k-t$ containing at most three non-stars.
By induction we have a $(k-t-1)$-edge-coloring of $G_{k-t}$ such that each tree in the new sequence is isomorphic to a subgraph spanned by the edges of a single color.
To complete the desired edge-coloring of $G$ we have two steps. First we color a few more edges to finish the non-stars in the original sequence. Second we introduce $t$ new colors and color edges of $G$ to get the deleted stars.
Generally the (easy) details of the second step are left to the reader.

Throughout the proof if we remove some vertices of a tree $T_i$
we denote the remaining graph by $T_i'$. Note that although $T_i$ denotes
a tree with $i$ vertices, $T_i'$ will always have fewer than $i$ vertices.

Let $x$ be a vertex in the tree $T$. After the inductive step we have an isomorphism between $T$ and a (monochromatic) subgraph of $G$. For simplicity the image of $x$ in $G$ will also be called $x$.

\begin{claim}\label{c1}
If $T_k$ is a star and $k \geq 3$, then $T_2, \dots, T_k$ have a packing into $G$.
\end{claim}

\begin{proof}
By induction there is a ($k-2$)-edge-coloring of $G_{k-1}$ such that each tree $T_2, \dots, T_{k-1}$ is isomorphic to a subgraph spanned by the edges of a single color.
There is at least one vertex $a$ in
$A_1$ and its degree is at least $k-1$ in $G$.
Thus we can color $k-1$ edges incident to $a$ with a new color to complete the edge-coloring of $G$.
\end{proof}

\begin{claim}\label{c2}
If $T_{k-1}$ is a star $k \geq 3$, then $T_2, \dots, T_k$ have a packing into $G$.
\end{claim}

\begin{proof}
Remove a leaf with neighbor $u$ from $T_k$
and let $T'_k$ be the resulting graph.
By induction there is a ($k-2$)-edge-coloring of $G_{k-1}$ such that each tree $T_2, \dots, T_{k-2},T'_k$ is isomorphic to a subgraph spanned by the edges of a single color.
We will call the color of $T_k'$ \emph{blue}.
The vertex $u$ has a neighbor $a \in A_1$. We color the edge $ua$ blue
to get a blue $T_k$.
The degree of $a$ is at least $k-1$ and $ua$ is the only colored edge incident to $a$. Thus $a$ has at least $k-2$ uncolored incident edges.
We color $k-2$ of these edges with a new color to get a monochromatic $T_{k-2}$.
This completes the edge-coloring of $G$.
\end{proof}

Note that Claim \ref{c2} implies the theorem for $k = 4$.

\begin{claim}\label{c3}
If $T_k$ and $T_{k-1}$ are not stars and $T_{k-2}$ and $T_{k-3}$ are
both stars and $k \geq 5$, then $T_2, \dots, T_k$ have a packing into $G$.
\end{claim}

\begin{proof}
The trees $T_k$ and $T_{k-1}$ are not stars so we can
remove two leaves with neighbors $u$ and $v$ from $T_k$ and two
leaves with neighbors $x$ and $y$ from $T_{k-1}$ such that $u\neq
v$ and $x\neq y$ and let $T_{k}'$ and $T_{k-1}'$ be the remaining graphs.

By induction there is a ($k-3$)-edge-coloring of $G_{k-2}$ such that each tree $T_2, \dots, T_{k-4}, T_{k-1}',T_k'$ is isomorphic to a subgraph spanned by the edges of a single color. We will call the color of $T_k'$ and $T_{k-1}'$ \emph{blue} and \emph{red} respectively.

Without loss of generality we can
suppose that $u\neq x$ and $v\neq y$ in $G$. There is a neighbor
$a \in A_1$ of $u$, a neighbor $b \in A_2$ of $v$, a neighbor
$a' \in A_1$ of $x$ and a neighbor $b' \in A_2$ of $y$.

We color the edges $ua$ and $vb$ blue to get a blue $T_k$. We color the edges $xa'$ and
$yb'$ red to get a red $T_{k-1}$.
The vertex $a$ is incident to at least $k-3$
uncolored edges.
We color $k-3$ of these edges with a new color to get a monochromatic $T_{k-2}$.
Now the vertex $b$ is incident to at least $k-4$
uncolored edges. We color $k-4$ of these edges with another new color to get a monochromatic $T_{k-3}$.
This completes the edge-coloring of $G$.
\end{proof}

Note that Claim \ref{c3} implies the theorem for $k=5$. Furthermore,
the above three claims are essentially the same as the proof of
the first theorem in \cite{GyLe}.

\begin{claim}\label{c4}
If there is a pending star $R$ of order $r$ in $T_k$ and $T_{k-r}$ is a star,
then $T_2, \dots, T_k$ have a packing into $G$.
\end{claim}

\begin{proof}
Let $u$ be the neighbor of $R$.
Remove $R$ from $T_k$ and let $T_k'$ be the
remaining graph.

By induction there is a ($k-2$)-edge-coloring of $G_{k-1}$ such that each tree $T_2, \dots, T_{k-r-1}, T_k',
T_{k-r+1}, \dots, T_{k-1}$ is isomorphic to a subgraph spanned by the edges of a single color. We will call the color of $T_k'$ \emph{blue}.

There is a neighbor $a \in A_1$
of $u$ and $a$ has at least $k-2$ other neighbors in $G$.
Then there are at least $r-1$ vertices $d_1, \dots, d_{r-1}$ which
are neighbors of $a$ but are not in $T_k'$ i.e. there are no blue edges incident to $d_1,\dots, d_{r-1}$.
We color the edges $ua$
and $ad_1,\dots, ad_{r-1}$ blue to get a blue $T_k$.
Now the vertex $a$ is incident to at least $k-r-1$ uncolored edges.
We color $k-r-1$ of these edges with a new color to get a monochromatic $T_{k-r}$.
This completes the edge-coloring of $G$.
\end{proof}

% a,b,c,d, a', a_1, etc refer to verts in A_1, A_2, etc.
% u,v are neighbors in T_k and x,y in T_k-1

\begin{claim}
For $2 \leq i \leq k \leq 6$, let $T_i$ be a tree on $i$ vertices.
If $G$ is a $k$-chromatic graph, then $T_2,\dots,T_k$ can be packed into $G$.
\end{claim}

\begin{proof}
By the above claims, the only remaining case is when $k=6$ and none of $T_6, T_5, T_4$ are stars. It is easy to see that $T_2, T_3, T_4$ are unique (they are all paths) and $T_5$ and $T_6$ each have two possible configurations (either a path or a spider).

Remove a pending star of order $2$ from $T_6$ and a leaf from $T_4$ and let $T_6'$ and $T_4'$ be the remaining graphs. Note that both of these remaining graphs are paths.
We can reconstruct $T_4$ by adding an edge to either endpoint of $T_4'$. Similarly, we can reconstruct $T_6$ by adding a pending star of order $2$ to either endpoint (if $T_6$ is a path) or to either interior point (if $T_6$ is a spider).

Because the statement of the claim holds for $k=5$ there is a $4$-edge-coloring of $G_{5}$ such that each tree $T_2, T_4', T_6',T_5$ is isomorphic to a subgraph spanned by the edges of a single color. We will call the color of $T_6'$ and $T_{4}'$ \emph{blue} and \emph{red} respectively.

First we consider the case when $T_6$ is a spider (if $T_6$ is a path, then the argument below works if we replace ``interior point'' with ``endpoint'' everywhere). We distinguish two subcases.

\vspace{.2cm}
{\bf Case A.} The two endpoints $u$ and $v$ of $T_4'$ are equal to the two interior points of $T_6'$ in $G_5$.
\vspace{.2cm}

The vertex $u$ has a neighbor $a \in A_1$ and $a$ has a neighbor $d_1 \in G$ which is not in $T_6'$. We color $ua$ and $ad_1$ blue to get a monochromatic $T_6$. The vertex $v$ has a neighbor $b \in A_1$ (note that $b$ and $a$ can be the same vertex, but still the edge $vb$ is uncolored). We color $vb$ red to get a monochromatics $T_4$. Now there are at least two uncolored edges incident to $a$. We can color them with a new color to get a monochromatic $T_3$ to complete the edge-coloring of $G$.

\vspace{.2cm}
{\bf Case B.} There is an interior point $u$ of $T_6'$ which is not an endpoint of $T_4'$.
\vspace{.2cm}

The vertex $u$ has a neighbor $a \in A_1$ and $a$ has a neighbor $d_1 \in G$ which is not in $T_6'$. We color $ua$ and $ad_1$ blue to get a monochromatic $T_6$. One of the endpoints $v$ of $T_4'$ is not equal to $d_1$. The vertex $v$ has a neighbor $b \in A_1$ (note that $b$ and $a$ can be the same vertex, but still the edge $vb$ is uncolored). We color $vb$ red to get a monochromatics $T_4$. Now there are at least two uncolored edges incident to $a$. We can color them with a new color to get a monochromatic $T_3$ to complete the edge-coloring of $G$.
\end{proof}

From now on we can suppose that none of the conditions of the above five claims hold.
In particular, $k > 6$ and $T_k$, $T_{k-1}$ plus exactly one of
$T_{k-2}$ and $T_{k-3}$ are not stars. Thus all other trees are
stars. Furthermore, all the pending stars in
$T_k$ have order $2$ (in the case $T_{k-2}$
is not a star) or order $3$ (in the case $T_{k-3}$ is not a star).

We now distinguish two cases and several subcases.

\vspace{.2cm}
{\bf Case 1.} Every pending star in $T_k$ is of order $3$ (i.e. the case $T_{k-3}$ is not a star).
\vspace{.2cm}

Let $R$ be a pending star of order $3$ in $T_k$ with neighbor $u$ and let
$v$ be the neighbor of a leaf such that $u \neq v$ and $v$ is not in $R$ (such a leaf can be easily found as $k>6$).
Let $x$ be the neighbor of a leaf in $T_{k-1}$.
Remove $R$ and a leaf which is a neighbor of $v$ from $T_k$ and let $T_k'$ be the
remaining graph. Remove a leaf which is a neighbor of $x$ from $T_{k-1}$ and let
$T_{k-1}'$ be the remaining graph.

By induction there is a ($k-3$)-edge-coloring of $G_{k-2}$ such that each tree $T_2, \dots, T_{k-5}, T_{k}',T_{k-3},T_{k-1}'$ is isomorphic to a subgraph spanned by the edges of a single color. We will call the color of $T_k'$ and $T_{k-1}'$ \emph{blue} and \emph{red} respectively.

The vertex $v$ has a neighbor $a \in A_1$ and $u$ has a
neighbor $b \in A_2$. There are at least $k-2$ neighbors of $b$
which are different from $a$. There are $k-4$ vertices in $T_k'$, hence there are at least two vertices
$d_1$ and $d_2$ adjacent to $b$ that are not in $T_k'$ and not equal to $a$.

The edges
$va$, $ub$, $bd_1$ and $bd_2$ are colored blue to get a blue $T_k$.
There are at least $k-2$ uncolored edges incident to $a$ and $k-4$ uncolored edges incident to $b$.
Although $x$ could coincide with $u$ or $v$, in any case there is at least one uncolored edge between $x$ and a vertex in $A_1$ or $A_2$.
We color this edge red to get a red $T_{k-1}$.
Then $a$ and $b$ have either
at least $k-3$ and $k-4$ or at least $k-2$ and $k-5$ uncolored
incident edges.
In either case, it is easy to see that we can color edges incident to $a$ or $b$ with two new colors to get $T_{k-2}$ and $T_{k-4}$ to complete the edge-coloring of $G$.

\vspace{.2cm}
{\bf Case 2.} Every pending star in $T_k$ is of order $2$ (i.e. the case $T_{k-2}$ is not a star).
\vspace{.2cm}

{\bf Case 2.1.} $T_k$ is not a spider.
\vspace{.2cm}

Let $R_1$ and $R_2$ be pending stars in $T_k$ of order $2$ with neighbors $u$ and $v$ such that $u \neq v$.
Let $x \neq y$ be neighbors of leaves in $T_{k-1}$.
Remove $R_1$ and $R_2$ from $T_k$ and let $T_k'$ be the remaining graph. Remove a leaf with neighbor $x$ and
a leaf with neighbor $y$ from $T_{k-1}$ and let $T_{k-1}'$ be the remaining graph.

By induction there is a ($k-3$)-edge-coloring of $G_{k-2}$ such that each tree $T_2, \dots, T_{k-5}, T_k', T_{k-1}',T_{k-2}$ is isomorphic to a subgraph spanned by the edges of a single color. We will call the color of $T_k'$ and $T_{k-1}'$ \emph{blue} and \emph{red} respectively.

Without loss of generality we
can suppose $u\neq x$ and $v\neq y$. There is a neighbor $a \in A_1$ of
$u$ and a neighbor $b \in A_2$ of $v$. There are at
least two vertices adjacent to $a$ and at least two vertices adjacent to
$b$ which are not in $T_k'$ and are different from $a$ and $b$. Thus we can find two vertices $d_1, d_2 \not \in T_k'$
such that $d_1$ is adjacent to $a$  and $d_2$ is
adjacent to $b$ and either $d_1\neq x$ and $d_2\neq y$ or $d_1=x$ and $d_2=y$.

Then the edges $ua$, $ad_1$, $vb$, $bd_2$ are colored blue to get a blue $T_k$.
Now there is an uncolored edge between $x$ and $A_1 \cup A_2$ and an uncolored edge between $y$ and
$A_1 \cup A_2$. Color these two edges red to get a red $T_{k-1}$
Now $a$ is incident to at least $k-4$
uncolored edges. We color $k-4$ of these edges with a new color to get a monochromatic $T_{k-3}$.
Now $b$ is incident to at least $k-5$
uncolored edges.
We color $k-5$ of these edges with another new color to get a monochromatic $T_{k-4}$.
This completes the edge-coloring of $G$.

\vspace{.2cm}
{\bf Case 2.2.} $T_k$ is a spider.
As $k > 6$, we can suppose that there exist three distinct vertices $u_1,u_2,u_3$ in $T_k$ each with at least one neighbor that is a leaf.
\vspace{.2cm}

{\bf Case 2.2.1.}  $T_{k-1}$ has a pending star $R$ of order $r\ge 4$.
\vspace{.2cm}

Let $x$ be the neighbor of $R$.
Let $w \neq z$ be neighbors of leaves in $T_{k-2}$.
Remove a neighboring leaf from each vertex $u_1,u_2,u_3$ in $T_k$ and let $T_k'$ be the remaining graph.
Remove the pending star $R$ from $T_{k-1}$ and let $T_{k-1}'$ be the remaining graph.
Remove a leaf with neighbor $w$ and
a leaf with neighbor $z$ from $T_{k-2}$ and let $T_{k-2}'$ be the remaining graph.

By induction there is a ($k-4$)-edge-coloring of $G_{k-3}$ such that each tree
$T_2, \dots, T_{k-r-2}, T_{k-1}',T_{k-r},\dots, T_{k-5}, T_{k-2}', T_k'$ is isomorphic to a subgraph spanned by the edges of a single color. We will call the color of $T_k'$, $T_{k-1}'$ and $T_{k-2}'$ \emph{blue}, \emph{red} and \emph{green} respectively.

The vertex $x$ has a neighbor
$a \in A_1$. There are at least $k-1$ neighbors of $a$ and
at least $k-1-(k-1-r)=r$ of them, say $d_1, \dots, d_{r}$ are not in
$T_{k-1}'$. If any of them is equal to $u_1,u_2,u_3$ then without loss of generality we can assume that
$d_{r}$ is equal to $u_3$ (other equalities are also possible).
If not, we still can suppose without loss of generality that $u_3\neq x$.
We color the edges $xa, ad_1, \dots, ad_{r-1}$ red to get a red $T_{k-1}$.

There is a neighbor $b \in A_2$ of $u_1$ and a neighbor $c \in A_3$ of
$u_2$ and there is an uncolored edge between $u_3$ and $A_1$.
Now we color the edges $u_1b,u_2c$ and the uncolored edge between $u_3$ and $A_1$ with color blue to get a blue
$T_k$.
It is easy to see that we can color either an edge between $w$ and
$A_2$ and an edge between $z$ and $A_3$, or an edge between $w$
and $A_3$ and an edge between $z$ and $A_2$ with color green to get a green $T_{k-2}$.

Now $a$ is incident to at least $k-r-2$ uncolored edges, so we can color edges incident to $a$ with a new color to get $T_{k-r-1}$. After this, $b$ and $c$ both are still incident to at least
$k-4$ uncolored edges (note that the edge $ba$ and $ca$ may be colored red or with the new color corresponding to $T_{k-1}$). It is easy to see that we can color edges incident to $b$ and $c$ with two new colors to get $T_{k-3}$ and $T_{k-4}$
to complete the edge-coloring of $G$.

\vspace{.2cm}
{\bf Case 2.2.2.}  $T_{k-1}$ has a pending star $R$ of order $3$.
\vspace{.2cm}

{\bf Case 2.2.2.1.} $T_{k-2}$ has a pending star $R'$ of order $r\ge 3$.
\vspace{.2cm}

Let $x$ be the neighbor of $R$.
Let $w$ be the neighbor of $R'$.
Remove a neighboring leaf from each vertex $u_1,u_2,u_3$ in $T_k$ and let $T_k'$ be the remaining graph.
Remove the pending star $R$ from $T_{k-1}$ and let $T_{k-1}'$ be the remaining graph.
Remove the pending star $R'$ from $T_{k-2}$ and let $T_{k-2}'$ be the remaining graph.

By induction there is a ($k-4$)-edge-coloring of $G_{k-3}$ such that each tree
$T_2, \dots, T_{k-r-3}, T_{k-2}',T_{k-r-1},\dots, T_{k-5}, T_{k-1}', T_k'$ is isomorphic to a subgraph spanned by the edges of a single color. We will call the color of $T_k'$, $T_{k-1}'$ and $T_{k-2}$ \emph{blue}, \emph{red} and \emph{green} respectively.

Without loss of generality we can suppose
$x\neq u_3$ and $w\neq u_2$. Then $x$ has a neighbor $c$ in
$A_3$. We color the edge $xc$, an edge between $c$ and $A_1$ and an
edge between $c$ and $A_2$ with color red to get a red $T_{k-1}$.
There is a neighbor $a \in A_1$
of $w$. The vertex $a$ has at least $k-2$ neighbors different
from $c$, at least $r$ of them, say $d_1, \dots, d_r$ are not in
$T_{k-2}'$.  If any $d_i$ is equal to $u_1,u_2,u_3$ then let
$d_{r}$ be equal to $u_2$ (other equalities are also possible).
We color the edges $wa$, $ad_1,\dots, ad_{r-1}$ green to get a green
$T_{k-2}$.

There is an uncolored edge between $u_3$ and $A_3$, an uncolored edge
between $u_2$ and $A_1$ and an uncolored edge between $u_1$ and $A_2$.
We color these edges blue to get a blue $T_{k}$.
It is easy to see that there are enough uncolored edges incident to $a$, $b$ and $c$ such that we can complete the edge-coloring of $G$ with three new colors to get $T_{k-r-2}$, $T_{k-4}$ and $T_{k-3}$.

\vspace{.2cm}
{\bf Case 2.2.2.2.} All pending stars in $T_{k-2}$ are of order $2$.
\vspace{.2cm}

Let $x$ be the neighbor of $R$. Let $R'$ be a pending star of order $2$ in $T_{k-2}$. Let $w$ be the neighbor of $R'$.
As $k > 6$, there is a leaf in $T_{k-2}$ with neighbor $z \neq w$.
Remove a neighboring leaf from each vertex $u_1,u_2,u_3$ in $T_k$ and let $T_k'$ be the remaining graph.
Remove the pending star $R$ from $T_{k-1}$ and let $T_{k-1}'$ be the remaining graph.
Remove the pending star $R'$ and a leaf with neighbor $z$ from $T_{k-2}$ and let $T_{k-2}'$ be the remaining graph.

By induction there is a ($k-4$)-edge-coloring of $G_{k-3}$ such that each tree
$T_2, \dots,T_{k-6}, T_{k-2}', T_{k-1}', T_k'$ is isomorphic to a subgraph spanned by the edges of a single color. We will call the color of $T_k'$, $T_{k-1}'$ and $T_{k-2}$ \emph{blue}, \emph{red} and \emph{green} respectively.

Without loss of generality we can suppose that $x \neq u_3$,
$w\neq u_2$ and $z\neq u_1$. Then $x$ has a neighbor $c \in A_3$. We
color the edge $xc$, an edge between $c$ and $A_1$ and an edge
between $c$ and $A_2$ with color red to get a red $T_{k-1}$.
There is a neighbor $a \in A_1$ of
$w$.
There is a neighbor $b \in A_2$ of $z$.
The vertex $a$ has at least $k-3$ neighbors different from
$c$ and $b$, at least two of them, $d_1$ and $d_2$, are not in
$T_{k-2}'$. Without loss of generality we can suppose that $u_2\ne d_1$.
We color the edges $wa$, $ad_1$ and $zb$ green to get a green $T_{k-2}$.
There is an edge between $u_3$ and $A_3$, an edge
between $u_2$ and $A_1$ and an edge between $u_1$ and $A_2$.
These edges are uncolored.
We color these edges blue to get a blue $T_k$.

It is easy to see that there are enough uncolored edges incident to $a$, $b$ and $c$ such that we can complete the edge-coloring of $G$ with three new colors to get $T_{k-5}$, $T_{k-4}$ and $T_{k-3}$.

\vspace{.2cm}
{\bf Case 2.2.3.} Every pending star in $T_{k-1}$ is of order $2$.
\vspace{.2cm}

{\bf Case 2.2.3.1.} $T_{k-1}$ is not a spider.
\vspace{.2cm}

Let $R$ and $R'$ be pending stars in $T_{k-1}$ of order $2$ with neighbors $x$ and $y$ such that $x \neq y$.
Let $w \neq z$ be neighbors of leaves in $T_{k-2}$.
Remove a neighboring leaf from each vertex $u_1,u_2,u_3$ in $T_k$ and let $T_k'$ be the remaining graph.
Remove $R$ and $R'$ from $T_{k-1}$ and let $T_{k-1}'$ be the remaining graph. Remove a leaf with neighbor $z$ and
a leaf with neighbor $w$ from $T_{k-2}$ and let $T_{k-2}'$ be the remaining graph.

By induction there is a ($k-4$)-edge-coloring of $G_{k-3}$ such that each tree
$T_2, \dots,T_{k-6}, T_{k-1}', T_{k-2}', T_k'$ is isomorphic to a subgraph spanned by the edges of a single color. We will call the color of $T_k'$, $T_{k-1}'$ and $T_{k-2}'$ \emph{blue}, \emph{red} and \emph{green} respectively.

There is a neighbor
$a \in A_1$ of $x$ and a neighbor $b \in A_2$ of $y$. There are
at least three neighbors $d_1,d_2,d_3$ of $a$ not in $T_{k-1}'$
and other than $b$ and at least three neighbors $f_1,f_2,f_3$ of
$b$ not in $T_{k-1}'$ and other than $a$. If there is a $u_i=d_j$
and/or $u_i=f_l$, we can suppose $j=3$ and/or $l=3$. Also we can suppose
$d_1\neq f_1$. We color the edges $xa$, $ad_1$, $yb$ and $bf_1$ red
to get a red $T_{k-1}$.

There are at most two of $u_1$, $u_2$ and $u_3$ equal to some of $x$, $y$,
$d_1$ and $f_1$, moreover, we can suppose that they are not $u_1$ and
$u_2$. Then there is a neighbor $c \in A_3$ of $u_3$. We color the
edge $u_3c$ blue. At most one of $u_1$ and $u_2$, say $u_1$
is connected by a red edge to a vertex in $A_1$ or $A_2$, say $A_1$,
then we color an edge between $u_1$ and $A_2$ and an edge between $u_2$ and $A_1$ with color blue to get a blue $T_k$.
Any vertex in $G_{k-3}$ is connected by a colored edge to at most two
of $A_1$, $A_2$ or $A_3$ and there are no two distinct vertices in $G_{k-3}$
connected by colored edges to the same two of $A_1$, $A_2$ or $A_3$.
Thus we can find an uncolored edge from $w$ to one class $A_1, A_2, A_3$ and an uncolored edge from
$z$ to a different class $A_1, A_2, A_3$. We color these two edges green to get a green $T_{k-2}$.

There are at least $k-5$, $k-4$ and $k-3$ or at least $k-5$, $k-5$ and $k-2$ uncolored
edges incident to $a$, $b$ and $c$ respectively.
It is easy to see that we can complete the edge-coloring of $G$ with three new colors to get $T_{k-5}$, $T_{k-4}$ and $T_{k-3}$.

\vspace{.2cm}
{\bf Case 2.2.3.2.} $T_{k-1}$ is a spider.
As $k> 6$, there exist three distinct vertices $x_1,x_2,x_3$ in $T_{k-1}$ each with at least one neighbor that is a leaf.
\vspace{.2cm}

{\bf Case 2.2.3.2.1.} $T_{k-2}$ has a pending star $R$ of order $r \ge 3$.
\vspace{.2cm}

Let $w$ be the neighbor of $R$.
Remove a neighboring leaf from each vertex $u_1,u_2,u_3$ in $T_k$ and let $T_k'$ be the remaining graph.
Remove a neighboring leaf from each vertex $x_1,x_2,x_3$ in $T_{k-1}$ and let $T_{k-1}'$ be the remaining graph.
Remove the pending star $R$ from $T_{k-2}$ and let $T_{k-2}'$ be the remaining graph.

By induction there is a ($k-4$)-edge-coloring of $G_{k-3}$ such that each tree
$T_2, \dots, T_{k-r-3}, T_{k-2}',T_{k-r-1},\dots, T_{k-5}, T_{k-1}', T_k'$ is isomorphic to a subgraph spanned by the edges of a single color. We will call the color of $T_k'$, $T_{k-1}'$ and $T_{k-2}'$ \emph{blue}, \emph{red} and \emph{green} respectively.

Without loss of generality we can suppose that $u_3, w, x_2$ are pairwise distinct.
There is a neighbor $a$ of $w$ in
$A_1$. There are at least $r+1$ neighbors $d_1, d_2 \dots,d_{r+1}$ of
$a$ not in $T_{k-2}'$.
If there is $u_i=d_j$, then we
can suppose $d_{r+1}=u_3$.  If $x_l=d_{r+1}$ also, then
we can suppose $l=3$. If one of $d_1, d_2,\dots, d_r$ is equal to some
$x_p$, then we can suppose $d_r=x_2$.

Now we color the edges $wa, ad_1, ad_2,\dots, d_{r-1}$ green to get a green $T_{k-2}$.
There is a neighbor $b \in A_2$ of $u_1$ and a neighbor $c \in A_3$ of
$u_2$. We color the edges $u_3a$, $u_1b$ and $u_2c$ blue to get a blue
$T_k$.
We color the edge $x_2a$ red. Now there is only one
colored edge in $G$ incident to $b$ and only one colored edge in $G$ incident to $c$. Furthermore, these two
colored edges are not incident.
Hence we can color either the
edges $x_1b$ and $x_3c$ or the edges $x_1c$ and $x_3b$ with color red to
get a red $T_{k-1}$.

Now there are at least $k-3$ uncolored edges incident to $b$, $k-3$
uncolored edges incident to $c$ and $k-r-3$ uncolored edges incident to $a$.
It is easy to see that we can complete the edge-coloring of $G$ with three new colors to get $T_{k-r-2}$, $T_{k-4}$ and $T_{k-3}$.

\vspace{.2cm}
{\bf Case 2.2.3.2.2.} Every pending star in $T_{k-2}$ is of order $2$.
\vspace{.2cm}

Let $w$ be the neighbor of a pending star $R$ in $T_{k-2}$ of order $2$.
Let $z \neq w$ be a neighbor of a leaf in $V(T_{k-2}) \setminus R$.
Remove a neighboring leaf from each vertex $u_1,u_2,u_3$ in $T_k$ and let $T_k'$ be the remaining graph.
Remove a neighboring leaf from each vertex $x_1,x_2,x_3$ in $T_{k-1}$ and let $T_{k-1}'$ be the remaining graph.
Remove the pending star $R$ and a leaf with neighbor $z$ from $T_{k-2}$ and let $T_{k-2}'$ be the remaining graph.

By induction there is a ($k-4$)-edge-coloring of $G_{k-3}$ such that each tree
$T_2, \dots,T_{k-6}, T_{k-2}', T_{k-1}', T_k'$ is isomorphic to a subgraph spanned by the edges of a single color. We will call the color of $T_k'$, $T_{k-1}'$ and $T_{k-2}'$ \emph{blue}, \emph{red} and \emph{green} respectively.

Without loss of generality we can suppose that $x_1 \neq w$.
There is a neighbor
$a \in A_1$ of $z$ and a neighbor $b \in A_2$ of $w$. There are
at least three neighbors $d_1,d_2,d_3$ of $b$ which are not in
$T_{k-2}'$ and are different from $a$. We can suppose that if $z$ is equal to some $u_i$ and/or
some $x_j$, then $z=u_1$ and/or $z=x_2$. We also can suppose that
$u_1$, $x_1$ and $d_1$ are pairwise distinct.

We color the edges $za$, $wb$ and $bd_1$ green to get a green $T_{k-2}$. Then we color
an edge between $u_1$ and $A_2$ blue. We color an edge between
$x_1$ and $A_2$ red. Now none of $u_2$, $u_3$, $x_3$ has a colored edge
incident to $A_1$ or $A_3$, but it is possible that there is a colored edge between $x_2$ and $a$.
We can color an edge from $x_2$ to $A_3$ and an edge from $x_3$ to
$A_1$ red to get a red $T_{k-1}$.
Now we can color either the edges from $u_2$ to $A_1$ and
from $u_3$ to $A_3$ or the edges from $u_2$ to $A_3$ and from $u_3$
to $A_1$ with color blue to get a blue $T_{k}$.

Now there are at least $k-4$ uncolored edges incident to $a$, $k-5$
uncolored edges incident $b$ and $k-3$ uncolored edges incident to some $c \in A_3$.

It is easy to see that we can complete the edge-coloring of $G$ with three new colors to get $T_{k-5}$, $T_{k-4}$ and $T_{k-3}$.
\end{proof}

\section{Additional conjectures and results}

In this section we prove simple propositions for tree packings into graphs with minimum or average degree conditions. We also introduce some additional conjectures.

In the case of $k$-chromatic graphs, we could assume that the minimum degree is at least $k-1$.
This suggests the following generalization of Conjecture \ref{chromatic-conj}.

\begin{conj}\label{mindeg}
For $2 \leq i \leq k$, let $T_i$ be a tree on $i$ vertices.
If a graph $G$ has minimum degree $\delta(G)\ge k-1$,
then the set of trees $T_2, \dots, T_k$ has a packing into $G$.
\end{conj}

When the number of vertices of $G$ is large with respect to the minimum degree, then
Conjecture \ref{mindeg} is true:

\begin{prop}\label{mindeg-prop}
For $2 \leq i \leq k$, let $T_i$ be a tree on $i$ vertices. There is a constant $n_0(k)$ such that
if $G$ is a graph on $n > n_0(k)$ vertices and minimum degree $\delta(G)\ge k-1$, then
$T_2,\dots, T_k$ can be packed into $G$.
\end{prop}

This proposition is an easy corollary of the following lemma. Indeed, by the lemma we can find and remove one by one all the required trees.

\begin{lem}
There is a constant $n_0(k)$ such that if $G$ is a graph on $n>n_0(k)$ vertices and minimum degree
$\delta(G)\ge k-1$, and $G'$ is the graph remaining after
removing an arbitrary set of ${k \choose 2}$ edges from $G$, then any
tree on $k$ vertices is a subgraph of $G'$.
\end{lem}

\begin{proof}
Let $B_1$ be the set of vertices with degree less than $k-1$ in
$G'$. Let $B_2\subset V(G')\setminus B_1$ be the neighbors of
$B_1$ adjacent to less than $k-1$ vertices in $V(G')\setminus
B_1$. For $2 < i \leq k$, let $B_i\subset V(G')\setminus\cup_{j<i}
B_j$ be the neighbors of $\cup_{j<i} B_j$  adjacent to less than
$k-1$ vertices in $V(G') \setminus \cup_{j<i} B_j$. Finally let
$B=\cup_{i\le k}B_{i}$. Note that in each step $i$, each vertex in $\cup_{j<i} B_j$ is adjacent to less than $k-1$ vertices of $V(G')\setminus\cup_{j<i} B_j$.

Clearly $|B_1| \leq 2{k \choose 2}= k^2-k$. Then $|B_2|\le (k-1)|B_1|$ as each
vertex in $B_2$ is a neighbor of some vertex in $B_1$. For $2 < i \leq
k$, by the same argument we have $|B_i| \le (k-1)|\cup_{j<i}B_{j}|$, thus $|\cup_{j\le i}B_{j}|\le k|\cup_{j<i}B_{j}|$. So $|B|
= |\cup_{j\le k}B_{j}| \le k^{k-1}(k^2-k) $ and thus the cardinality of $B$
does not depend on $n$. Choose $n_0(k)$ to be bigger then this constant, this way there is a vertex not in $B$.

Choose an arbitrary vertex of the tree as a root and note that
each vertex has a fixed distance in the tree from the root, which is at most
$k-1$. We denote by \emph{level $i$} the set of vertices of the tree of distance
$i$ from the root. Identify the root with a vertex in
$V(G') \setminus B$. There are at most $k-1$ vertices in level $1$
and at least $k-1$ neighbors of the root in $V(G') \setminus (B_1
\cup B_2 \cup \dots \cup B_{k-1})$ so we can identify the vertices
in level 1 with the neighbors of the root in $V(G') \setminus (B_1
\cup B_2 \cup \dots \cup B_{k-1})$. Similarly by induction we can identify
vertices in level $i$ with vertices of distance $i$ from the root
in $V(G')\setminus (B_1 \cup B_2 \cup \dots \cup B_{k-i})$. Indeed, suppose we have identified levels $1$ through $i$ with vertices in $G'$. Denote by $V_i$ the vertices of $G'$ that are identified with vertices of level $i$ of the tree. Each vertex of $V_i$ has at least $k-1$ adjacent vertices in $V(G')\setminus (B_1 \cup B_2 \cup \dots \cup B_{k-i-1})$.
Since the order of the tree is $k$, and since at least one vertex of the tree is already identified with vertices in $G'$, we can easily identify the vertices of level $i+1$ with vertices that have not yet been used in the previous steps.
\end{proof}

The bound on $n$ given by the proof of Proposition \ref{mindeg-prop} can
probably be improved. However, it seems unlikely that Conjecture \ref{chromatic-conj} can be proved with this type of argument.

We can weaken the minimum degree condition in Conjecture \ref{mindeg} to get an even stronger conjecture.

\begin{conj}\label{avgdeg}
For $2 \leq i \leq k$, let $T_i$ be a tree on $i$ vertices.
If the graph $G$ has average degree at least $k-1$, i.e. $G$ has at least $\frac{k-1}{2}n$ edges,
then the set of trees $T_2, \dots, T_k$ has a packing into $G$.
\end{conj}

In this setting it is easy to prove an analogue of the previously-mentioned result of Bollob\'as \cite{Bo}.

\begin{prop}\label{avg-deg-thm}
Given a fixed $s\le k/2$, for $2 \leq i \leq s$, let $T_i$ be a tree on $i$ vertices.
If $G$ is a graph with $n$ vertices and at least $\frac{k-1}{2}n$ edges where $k \le n$, then the set of trees $T_2 \dots, T_s$ has a packing into $G$.
\end{prop}

\begin{proof}
We proceed by induction on $k$. For $k=1$ the statement of the proposition obviously holds.
Now let us assume that $k > 1$ and the statement of the proposition holds for all values less than $k$.

Let $G$ be the graph in the statement of the proposition. Remove  all
vertices of $G$ with degree less than $\frac{k-1}{2}$. Let us
continue to remove all vertices with degree less than
$\frac{k-1}{2}$ from the resulting graphs until the procedure
stops. In each round the average degree cannot decrease, so when
the procedure stops we are left with a graph with minimum degree
at least $\frac{k-1}{2}$. Thus $G$ contains a subgraph with
minimum degree at least $\frac{k-1}{2} \geq \frac{k}{2}-1$.

It is easy to see that any tree $T_i$ is a subgraph (i.e. has a packing) into a graph with
minimum degree $i-1$.
Thus $T_s$ has a packing into $G$ as $s \leq \frac{k}{2}$
and the minimum degree of $G$ is at least $\frac{k}{2}-1$.
If we remove the edges of $T_s$ from $G$ we are left with a
graph with at least $\frac{k-1}{2}n - (\frac{k}{2}-1) \geq \frac{k-2}{2}n$ edges as $k \le n$
and we are done by induction.
\end{proof}

Moreover, from the proof it is easy to see that
if we have a graph as in the statement of Proposition \ref{avg-deg-thm}
then any packing of $T_i, \dots, T_s$ into $G$
can be extended to a packing $T_2, \dots, T_{s}$
using the remaining edges of $G$.

Conjecture \ref{avgdeg} is strongly related to the following
conjecture of Erd\H os and S\' os \cite{Er}.

\begin{conj}[Erd\H os and S\' os \cite{Er}]
Let $T_k$ be a tree with $k$ vertices.
If $G$ is a graph with $n$ vertices and more than
$\frac{k-2}{2}n$ edges, then $T_k$ is a subgraph of $G$.
\end{conj}

At first glance, Conjecture \ref{avgdeg} seems to ask for much more
as we have only a few more edges but we want to pack many more trees.
However, for graphs $G$ where $n \geq 2k$, if true, the Erd\H os-S\'os Conjecture easily implies
Conjecture \ref{avgdeg}.

In particular, let $G$ be a graph given in Conjecture \ref{avgdeg}
and let us assume that the Erd\H os-S\'os Conjecture is true. Then
$T_k$ is a subgraph of $G$ as $G$ has more than $\frac{k-2}{2}n$
edges. Removing $T_k$ from $G$ yields a graph $G'$ with $e(G')
\geq \frac{k-1}{2}n - (k-1) > \frac{k-2}{2}n$ as $n \geq 2k$. Thus by the Erd\H os-S\'os Conjecture, $G'$ has $T_{k-1}$ as a subgraph. This argument can be continued to find all the trees required by
Conjecture \ref{avgdeg}.

In this paper we proved most of the known results concerning the TPC in the more general setting where we pack trees into any $k$-chromatic graph.
However, missing from the more general setting is the analogue of the result of Gy\'arf\'as and Lehel \cite{GyLe} that states that the trees can be packed into $K_n$ if each tree is a path or a star.

\end{document}